\documentclass{amsart}
 \newtheorem{thm}{Theorem}[section]
 
 \newtheorem{lem}[thm]{Lemma}

 \theoremstyle{definition}
 \newtheorem{defn}[thm]{Definition}
 \theoremstyle{remark}

 \numberwithin{equation}{section}

\begin{document}

%
%
%
%
%
%
%
%
%

\title[Arithmetic properties of cubic and overcubic partition pairs]
 {Arithmetic properties of cubic and overcubic partition pairs}

\author{Chiranjit Ray}

\address{Department of Mathematics, Indian Institute of Technology Guwahati, Assam, India, PIN- 781039}
\email{chiranjitray.m@gmail.com}

\author{Rupam Barman}
\address{Department of Mathematics, Indian Institute of Technology Guwahati, Assam, India, PIN- 781039}
\email{rupam@iitg.ac.in}
\thanks{We are very grateful to Professor Ken Ono for careful reading of a draft of the
manuscript. The first author acknowledges the financial support of Department of Atomic Energy, Government of India for supporting a
	part of this work under NBHM Fellowship.}

\keywords{Cubic partition pair; overcubic partition; overcubic partition pair; modular forms}

\date{Auguest 10, 2018}
\dedicatory{}

\begin{abstract}
Let $b(n)$ denote the number of cubic partition pairs of $n$. We give affirmative answer to a conjecture of Lin, namely, we prove that $$b(49n+37)\equiv 0 \pmod{49}.$$
We also prove two congruences modulo $256$ satisfied by $\overline{b}(n)$, the number of overcubic partition pairs of $n$. Let $\overline{a}(n)$ denote the 
number of overcubic partition of $n$. For a fixed positive integer $k$, we further 
show that $\overline{b}(n)$ and $\overline{a}(n)$ are divisible by $2^k$ for almost all $n$. We use arithmetic properties of modular forms to prove our results.
\end{abstract}

\maketitle
\section{Introduction and statement of results}
In a series of papers \cite{chan1, chan2, chan3}, Chan studied the cubic partition function $a(n)$ with generating function given by
\begin{align*}
	\sum_{n=0}^{\infty}a(n)q^n &=\frac{1}{(q; q)_{\infty}(q^2; q^2)_{\infty}}, ~|q|<1,
\end{align*}
where $(a; q)_{\infty}:=\prod_{n\geq 0}(1-aq^n)$. The partition function $a(n)$ satisfies many interesting congruences. For example, it satisfies the 
following Ramanujan-like congruence
\begin{align*}
	a(3n+2)\equiv 0 \pmod{3}.
\end{align*}
Inspired by Chan's work, Zhao and Zhong \cite{zhao} studied the cubic partition pair function $b(n)$ which is defined by 
\begin{align*}
	\sum_{n=0}^{\infty}{b}(n)q^n&=\frac{1}{(q; q)_{\infty}^2(q^2; q^2)_{\infty}^2}.
\end{align*}
They established several Ramanujan-like congruences for $b(n)$ as follows:
\begin{align*}
 b(5n + 4) & \equiv 0 \pmod{5},\\
 b(7n + i) & \equiv 0 \pmod{7},\\
 b(9n + 7)& \equiv 0 \pmod{9},\\
\end{align*}
where $i=2, 3, 4, 6$. Recently, Lin \cite{lin2} studied the arithmetic properties of $b(n)$ modulo $27$. He also conjectured the following four congruences:
\begin{align}
\label{linconj-1} b(49n + 37) & \equiv 0 \pmod{49},\\
\label{linconj-2} b(81n + 61) & \equiv 0 \pmod{243},\\
\label{linconj-3} \sum_{n\geq 0}b(81n + 7)q^n & \equiv 9\frac{(q^2; q^2)_{\infty}(q^3; q^3)_{\infty}^2}{(q^6; q^6)_{\infty}} \pmod{81},\\
\label{linconj-4} \sum_{n\geq 0}b(81n + 34)q^n & \equiv 36\frac{(q; q)_{\infty}(q^6; q^6)_{\infty}^2}{(q^3; q^3)_{\infty}} \pmod{81}.
\end{align}
In two recent papers, Lin, Wang, and Xia \cite{lin3} and Chern \cite{chern}  independently proved \eqref{linconj-2}, \eqref{linconj-3} and \eqref{linconj-4}. 
In both the articles, it was proved that the congruence \eqref{linconj-2}
is in fact true modulo $729$. Recently, Hirschhorn \cite{hirs} also proved the congruence \eqref{linconj-2} modulo $729$. 
However, to the best of our knowledge the congruence \eqref{linconj-1} has not been established till date. 
\par In this paper, we prove that the Lin's conjecture \eqref{linconj-1} is true. 
\begin{thm} \label{thm1} For any non-negative integers $n$, we have
	\begin{align*}
	 b(49n+37)\equiv0\pmod{49}.
	\end{align*}
\end{thm}
We have $b(37)=80832850\not\equiv 0\pmod{7^3}$. Hence, unlike to Lin's conjecture \eqref{linconj-2}, the congruence \eqref{linconj-1} is best possible 
in the sense that the moduli cannot be replaced by higher power of $7$ such that the congruence holds for all $n\geq 0$.
\par 
In \cite{kim}, Kim introduced a partition function $\overline{b}(n)$ whose generating function is given by 
\begin{align}\label{gen-ocp}
	\overline{B}(q):=\sum_{n=0}^{\infty}\overline{b}(n)q^n&=\frac{(-q; q)_{\infty}^2(-q^2; q^2)_{\infty}^2}{(q; q)_{\infty}^2(q^2; q^2)_{\infty}^2}.
\end{align}
Kim named $\overline{b}(n)$ as the number of overcubic partition pairs of $n$. Using arithmetic properties of quadratic forms and modular forms, Kim \cite{kim} 
derived the following two congruences
\begin{align*}
 \overline{b}(8n + 7) &\equiv 0\pmod{64},\\
  \overline{b}(9n + 3) &\equiv 0\pmod{3}.
\end{align*}
In \cite{lin1}, Lin proved two Ramanujan-like congruences and several infinite families of congruences modulo $3$ satisfied by $\overline{b}(n)$. He also
obtained some congruences for $\overline{b}(n)$ modulo $5$. 
\par In this paper, we prove the following two congruences modulo $256$ satisfied by $\overline{b}(n)$.
\begin{thm} \label{thm2} For any non-negative integers $n$, we have
	\begin{align*}
	\overline{b}(72n+t)&\equiv~0\pmod{256},
	\end{align*}
	where $t\in \{42,66\}.$
\end{thm}
These two congruences are best possible since
\begin{align*}
\overline{b}(72+42) &= \overline{b}(114) = 173333430318331391232 \not \equiv 0 \pmod{512},\\
\overline{b}(66) &= 407868414339840 \not \equiv 0 \pmod{512}.
\end{align*}
\par For any fixed positive integer $k$, Gordon and Ono \cite{ono2} proved that the number of partitions of $n$ into distinct parts is divisible by $2^k$ for almost 
all $n$. Bringmann and Lovejoy \cite{bringman} showed that the number of overpartition pairs of $n$ is divisible by $2^k$ for almost all $n$. In \cite{lin4}, Lin 
proved that the number of overpartition pairs of $n$ into odd parts is also divisible by $2^k$ for almost all $n$.
\par In this article, we prove that the number of overcubic partition pairs of $n$ is divisible by $2^k$ for almost all $n$. 
\begin{thm} \label{thm3} Let $k$ be a positive integer. Then $\overline{b}(n)$ is almost always divisible by $2^k$, namely,
		\begin{align*}
	 \lim_{X\to\infty} \frac{\# \left\{n\leq X: \overline{b}(n)\equiv 0\pmod{2^k}\right\}}{X}=1.
	\end{align*}
\end{thm}
In \cite{kim2}, Kim studied the overpartion analog of cubic partition function. He defined the overcubic partition function 
$\overline{a}(n)$ whose generating function is given by 
\begin{align}\label{gen-oc}
	\overline{A}(q):=\sum_{n=0}^{\infty}\overline{a}(n)q^n&=\frac{(-q; q)_{\infty}(-q^2; q^2)_{\infty}}{(q; q)_{\infty}(q^2; q^2)_{\infty}}.
\end{align}
We also prove that the number of overcubic partitions of $n$ is divisible by $2^k$ for almost all $n$. 
\begin{thm} \label{thm4} Let $k$ be a positive integer. Then $\overline{a}(n)$ is almost always divisible by $2^k$, namely,
		\begin{align*}
	\lim_{X\to\infty} \frac{\# \left\{n\leq X: \overline{a}(n)\equiv 0\pmod{2^k}\right\}}{X}=1.
	\end{align*}
\end{thm}
\section{Proof of Theorems \ref{thm1} and \ref{thm2}}
We define the following matrix groups: 
\begin{align*}
\Gamma & :=\left\{\begin{bmatrix}
a  &  b \\
c  &  d      
\end{bmatrix}: a, b, c, d \in \mathbb{Z}, ad-bc=1
\right\},\\
\Gamma_{\infty} & :=\left\{
\begin{bmatrix}
1  &  n \\
0  &  1      
\end{bmatrix} \in \Gamma : n\in \mathbb{Z}  \right\}.
\end{align*}
For a positive integer $N$, let 
$$\Gamma_{0}(N) :=\left\{
\begin{bmatrix}
a  &  b \\
c  &  d      
\end{bmatrix} \in \Gamma : c\equiv~0\pmod N \right\}.$$
The index of $\Gamma_{0}(N)$ in $\Gamma$ is
\begin{align*}
 [\Gamma : \Gamma_0(N)] = N\prod_{p|N}(1+p^{-1}), 
\end{align*}
where $p$ is a prime divisor of $N$.
\par 
We prove Theorems \ref{thm1} and \ref{thm2} using the approach developed in \cite{radu1, radu2}. We now recall some of the definitions and results 
from \cite{radu1, radu2} which will be used to prove our results. Also, see \cite{wang}. 
For a positive integer $M$, let $R(M)$ be the set of integer sequences $r=(r_\delta)_{\delta|M}$ indexed by the positive divisors of $M$. 
If $r \in R(M)$ and $1=\delta_1<\delta_2< \cdots <\delta_k=M$ 
are the positive divisors of $M$, we write $r=(r_{\delta_1}, \ldots, r_{\delta_k})$. Define $c_r(n)$ by 
\begin{align}
\sum_{n=0}^{\infty}c_r(n)q^n:=\prod_{\delta|M}(q^{\delta};q^{\delta})^{r_{\delta}}_{\infty}=\prod_{\delta|M}\prod_{n=1}^{\infty}(1-q^{n \delta})^{r_{\delta}}.
\end{align}
The approach to proving congruences for $c_r(n)$ developed by Radu \cite{radu1, radu2} reduces the number of coefficients that one must check as compared with the classical method which uses Sturm's bound alone.
\par 
Let $m$ be a positive integer. For any integer $s$, let $[s]_m$ denote the residue class of $s$ in $\mathbb{Z}_m:= \mathbb{Z}/ {m\mathbb{Z}}$. 
Let $\mathbb{Z}_m^{*}$ be the set of all invertible elements in $\mathbb{Z}_m$. Let $\mathbb{S}_m\subseteq\mathbb{Z}_m$  be the set of all squares in $\mathbb{Z}_m^{*}$. For $t\in\{0, 1, \ldots, m-1\}$
and $r \in R(M)$, we define a subset $P_{m,r}(t)\subseteq\{0, 1, \ldots, m-1\}$ by
\begin{align*}
P_{m,r}(t):=\left\{t': \exists [s]_{24m}\in \mathbb{S}_{24m} ~ \text{such} ~ \text{that} ~ t'\equiv ts+\frac{s-1}{24}\sum_{\delta|M}\delta r_\delta \pmod{m} \right\}.
\end{align*}
\begin{defn}
	Suppose $m, M$ and $N$ are positive integers, $r=(r_{\delta})\in R(M)$ and $t\in \{0, 1, \ldots, m-1\}$. Let $k=k(m):=\gcd(m^2-1,24)$ and write  
	\begin{align*}
	\prod_{\delta|M}\delta^{|r_{\delta}|}=2^s\cdot j,
	\end{align*}
	where $s$ and $j$  are nonnegative integers with $j$ odd. The set $\Delta^{*}$ consists of all tuples $(m, M, N, (r_{\delta}), t)$ satisfying these conditions and all of the following.
	\begin{enumerate}
		\item Each prime divisor of $m$ is also a divisor of $N$.
		\item $\delta|M$ implies $\delta|mN$ for every $\delta\geq1$ such that $r_{\delta} \neq 0$.
		\item $kN\sum_{\delta|M}r_{\delta} mN/\delta \equiv 0 \pmod{24}$.
		\item $kN\sum_{\delta|M}r_{\delta} \equiv 0 \pmod{8}$.  
		\item  $\frac{24m}{\gcd{(-24kt-k{\sum_{{\delta}|M}}{\delta r_{\delta}}},24m)}$ divides $N$.
		\item If $2|m$, then either $4|kN$ and $8|sN$ or $2|s$ and $8|(1-j)N$.
	\end{enumerate}
\end{defn}
Let $m, M, N$ be positive integers. For $\gamma=
\begin{bmatrix}
	a  &  b \\
	c  &  d     
\end{bmatrix} \in \Gamma$, $r\in R(M)$ and $r'\in R(N)$, set 
	\begin{align*}
	p_{m,r}(\gamma):=\min_{\lambda\in\{0, 1, \ldots, m-1\}}\frac{1}{24}\sum_{\delta|M}r_{\delta}\frac{\gcd^2(\delta a+ \delta k\lambda c, mc)}{\delta m}
	\end{align*}
and 
	\begin{align*}
	p_{r'}^{*}(\gamma):=\frac{1}{24}\sum_{\delta|N}r'_{\delta}\frac{\gcd^2(\delta, c)}{\delta}.
	\end{align*}
	
	\begin{lem}\label{lem1}\cite[Lemma 4.5]{radu1} Let $u$ be a positive integer, $(m, M, N, r=(r_{\delta}), t)\in\Delta^{*}$ and $r'=(r'_{\delta})\in R(N)$. 
	Let $\{\gamma_1,\gamma_2, \ldots, \gamma_n\}\subseteq \Gamma$ be a complete set of representatives of the double cosets of $\Gamma_{0}(N) \backslash \Gamma/ \Gamma_\infty$. 
	Assume that $p_{m,r}(\gamma_i)+p_{r'}^{*}(\gamma_i) \geq 0$ for all $1 \leq i \leq n$. Let $t_{min}=\min_{t' \in P_{m,r}(t)} t'$ and
	\begin{align*}
	\nu:= \frac{1}{24}\left\{ \left( \sum_{\delta|M}r_{\delta}+\sum_{\delta|N}r'_{\delta}\right)[\Gamma:\Gamma_{0}(N)] -\sum_{\delta|N} \delta r'_{\delta}\right\}-\frac{1}{24m}\sum_{\delta|M}\delta r_{\delta} 
	- \frac{ t_{min}}{m}.
 	\end{align*}	
	If the congruence $c_r(mn+t')\equiv0\pmod u$ holds for all $t' \in P_{m,r}(t)$ and $0\leq n\leq \lfloor\nu\rfloor$, then it holds for all $t'\in P_{m,r}(t)$ and $n\geq0$.
	\end{lem}
	To apply the above lemma, we need the following result which gives us a complete set of representatives of the double coset in  
	$\Gamma_{0}(N) \backslash \Gamma/ \Gamma_\infty$. 
	\begin{lem}\label{lem2}\cite[Lemma 4.3]{wang} If $N$ or $\frac{1}{2}N$ is a square-free integer, then
		\begin{align*}
		\bigcup_{\delta|N}\Gamma_0(N)\begin{bmatrix}
		1  &  0 \\
		\delta  &  1      
		\end{bmatrix}\Gamma_ {\infty}=\Gamma.
		\end{align*}
	\end{lem}
\begin{proof}[Proof of Theorem \ref{thm1}]
 We have
 \begin{align*}
  \sum_{n=0}^{\infty}{b}(n)q^n=\frac{1}{(q; q)_{\infty}^2(q^2; q^2)_{\infty}^2}.
 \end{align*}
Using binomial theorem, we have 
\begin{align*}
	\sum_{n=0}^{\infty}{b}(n)q^n=\frac{1}{(q; q)_{\infty}^2(q^2; q^2)_{\infty}^2}&\equiv \frac{(q; q)_{\infty}^{49}}{(q; q)_{\infty}^2(q^2; q^2)_{\infty}^2(q^7; q^7)_{\infty}^7}\\
	&\equiv \frac{(q; q)_{\infty}^{47}}{(q^2; q^2)_{\infty}^2(q^7; q^7)_{\infty}^7}\pmod{49}.
\end{align*}
 We choose $(m,M,N,r,t)=(49,14,14,(47,-2,-7,0),37)$ and it is easy to verify that $(m,M,N,r,t) \in \Delta^{*}$ and $P_{m,r}(t)=\{37\}$.
By Lemma \ref{lem2}, we know that $\left\{\begin{bmatrix}
	1  &  0 \\
	\delta  &  1      
	\end{bmatrix}:\delta|14 \right\}$ forms a complete set of double coset representatives of $\Gamma_{0}(N) \backslash \Gamma/ \Gamma_\infty$.
	Let $\gamma_{\delta}=\begin{bmatrix}
	1  &  0 \\
	\delta  &  1      
	\end{bmatrix}$. Let $r'=(12,0,0,0)\in R(14)$ and we use $Sage$ to verify that
	$p_{m,r}(\gamma_{\delta})+p_{r'}^{*}(\gamma_{\delta}) \geq 0$ for each $\delta | N$. We compute that the upper bound in Lemma \ref{lem1} is $\lfloor\nu\rfloor=48$. 
	Using $Mathematica$ we verify that
	$b(49n+37)\equiv0\pmod{49}$ for $n\leq48$. For example:
	    \begin{align*} b(&49\times48+37)\\
	    	=&2547081709226856352575091322454383557126791567591335\\&339788372343399013917787449842071480\\
		    =&49 \times 519812593719766602566345167847833379005467666855\\&37455914048415171408447301784690654520.
		\end{align*}
	Thus, by Lemma \ref{lem1}, we conclude that $b(49n+37)\equiv0\pmod{49}$ for any $n\geq0$.\\
\end{proof}
\begin{proof}[Proof of Theorem \ref{thm2}]
We first recall the following $2$-dissection formula from \cite[ p. 40, Entry 25]{berndt}:
	\begin{align}\label{two-dissection}
	\frac{1}{(q; q)_{\infty}^4}=\frac{(q^4; q^4)_{\infty}^{14}}{(q^2; q^2)_{\infty}^{14}(q^8; q^8)_{\infty}^4}
	+4q\frac{(q^4; q^4)_{\infty}^2(q^8; q^8)_{\infty}^4}{(q^2; q^2)_{\infty}^{10}}.
	\end{align}
Employing \eqref{two-dissection} into \eqref{gen-ocp} and using the fact that $(-q; q)_{\infty}=\displaystyle\frac{(q^2; q^2)_{\infty}}{(q; q)_{\infty}}$, we have
\begin{align*}
\sum_{n=0}^{\infty}\overline{b}(n)q^n&=\frac{(q^4; q^4)_{\infty}^2}{(q; q)_{\infty}^4(q^2; q^2)_{\infty}^2}\\
&=\frac{(q^4; q^4)_{\infty}^2}{(q^2; q^2)_{\infty}^2}\left( \frac{(q^4; q^4)_{\infty}^{14}}{(q^2; q^2)_{\infty}^{14}(q^8; q^8)_{\infty}^4}
+4q\frac{(q^4; q^4)_{\infty}^2(q^8; q^8)_{\infty}^4}{(q^2; q^2)_{\infty}^{10}}\right).
\end{align*}
Extracting the terms containing $2n$ and then using \eqref{two-dissection}, we obtain
\begin{align*}
&\sum_{n=0}^{\infty}\overline{b}(2n)q^n\\
&=\frac{(q^2; q^2)_{\infty}^{16}}{(q; q)_{\infty}^{16}(q^4; q^4)_{\infty}^4}\\
&=\frac{(q^4; q^4)_{\infty}^{52}}{(q^2; q^2)_{\infty}^{40}(q^8; q^8)_{\infty}^{16}}+16q\frac{(q^4; q^4)_{\infty}^{40}}{(q^2; q^2)_{\infty}^{36}(q^8; q^8)_{\infty}^{8}}+96q^2\frac{(q^4; q^4)_{\infty}^{28}}{(q^2; q^2)_{\infty}^{32}}\\
&+256q^3\frac{(q^4; q^4)_{\infty}^{16}(q^8; q^8)_{\infty}^{8}}{(q^2; q^2)_{\infty}^{28}}+256q^4\frac{(q^4; q^4)_{\infty}^{4}(q^8; q^8)_{\infty}^{16}}{(q^2; q^2)_{\infty}^{24}}\\
&\equiv\frac{(q^4; q^4)_{\infty}^{52}}{(q^2; q^2)_{\infty}^{40}(q^8; q^8)_{\infty}^{16}}+16q\frac{(q^4; q^4)_{\infty}^{40}}{(q^2; q^2)_{\infty}^{36}(q^8; q^8)_{\infty}^{8}}
+96q^2\frac{(q^4; q^4)_{\infty}^{28}}{(q^2; q^2)_{\infty}^{32}}\pmod{256}.
\end{align*}
Extracting the terms containing $2n+1$, we obtain, modulo $256$
\begin{align*}
\sum_{n=0}^{\infty}\overline{b}(4n+2)q^n
&\equiv 16\frac{(q^2; q^2)_{\infty}^{40}}{(q; q)_{\infty}^{36}(q^4; q^4)_{\infty}^{8}}\\
&=16\frac{(q^2; q^2)_{\infty}^{24}}{(q; q)_{\infty}^{4}(q^4; q^4)_{\infty}^{8}}\left(\frac{(q^2; q^2)_{\infty}^{16}}{(q; q)_{\infty}^{32}}\right)\\
&\equiv 16\frac{(q^2; q^2)_{\infty}^{24}}{(q; q)_{\infty}^{4}(q^4; q^4)_{\infty}^{8}}\\
&\equiv 16\frac{(q^2; q^2)_{\infty}^{24}}{(q^4; q^4)_{\infty}^{8}}\left( \frac{(q^4; q^4)_{\infty}^{14}}{(q^2; q^2)_{\infty}^{14}(q^8; q^8)_{\infty}^4}
+4q\frac{(q^4; q^4)_{\infty}^2(q^8; q^8)_{\infty}^4}{(q^2; q^2)_{\infty}^{10}}\right).
\end{align*}
Finally, extracting the terms containing $2n$, we obtain
\begin{align}\label{c}
\sum_{n=0}^{\infty}\overline{b}(8n+2)q^n
&\equiv 16\frac{(q; q)_{\infty}^{10}(q^2; q^2)_{\infty}^{6}}{(q^4; q^4)_{\infty}^{4}}\pmod{256}.
\end{align}
\par Now, we choose $(m,M,N,r,t)=(9,8,12,(10,6,-4,0),5)$ and it is easy to verify that $(m,M,N,r,t) \in \Delta^{*}$ and $P_{m,r}(t)=\{5,8\}$.
   From  Lemma \ref{lem2} we know that $\left\{\begin{bmatrix}
	1  &  0 \\
	\delta  &  1      
	\end{bmatrix}:\delta|12 \right\}$ forms a complete set of double coset representatives of $\Gamma_{0}(N) \backslash \Gamma/ \Gamma_\infty$.
	Let $r'=(0,0,0,0,0,0)\in R(12)$ and we use $Sage$ to verify that
	$p_{m,r}(\gamma_{\delta})+p_{r'}^{*}(\gamma_{\delta}) \geq 0$ for each $\delta | N$, where $\gamma_{\delta}=\begin{bmatrix}
	1  &  0 \\
	\delta  &  1      
	\end{bmatrix}$. We compute that the upper bound in Lemma \ref{lem1} is $\lfloor\nu\rfloor=11$. 
	Using $Mathematica$ we verify that
	$\overline{b}(72n+42) \equiv 0 \pmod{256}$  and $\overline{b}(72n+66) \equiv 0 \pmod{256}$ for $n \leq 11$. For example:

		\begin{align*} 
			&\overline{b}(72\times11+42)\\
			&=7761203316486412630224281570704445214150462692805613719833344\\
			&= 256 \times 30317200455025049336813599885564239117775244893771928593099
		\end{align*}
		and
	\begin{align*} 
	&\overline{b}(72\times11+66)\\
	&=519085767900724771159507478715308091342334377184238781729434112\\
		&= 256 \times 2027678780862206137341826088731672231805993660875932741130602.
	\end{align*}
	Thus, by Lemma \ref{lem1}, we conclude that  $\overline{b}(72n+t)\equiv0\pmod{256}$ for any $n\geq0$, where $t\in \{42,66\}$. This completes the proof of the theorem.
	\end{proof}
	
\section{Proof of Theorems \ref{thm3} and \ref{thm4}}
Recall that the Dedekind's eta-function $\eta(z)$ is defined by
\begin{align*}
	\eta(z):=q^{1/24}(q;q)_{\infty}=q^{1/24}\prod_{n=1}^{\infty}(1-q^n),
\end{align*}
where $q:=e^{2\pi iz}$ and $z$ is in the upper half complex plane. A function $f(z)$ is called an eta-quotient if it is of the form
\begin{align*}
f(z)=\prod_{\delta|N}\eta(\delta z)^{r_\delta},
\end{align*}
where $N$ is a positive integer and $r_{\delta}$ is an integer. We now recall two theorems from \cite[p. 18]{ono} which will be used to prove our result.

\begin{thm}\cite[Theorem 1.64 and Theorem 1.65]{ono}\label{thm_ono1} If $f(z)=\prod_{\delta|N}\eta(\delta z)^{r_\delta}$ 
is an eta-quotient with $\ell=\frac{1}{2}\sum_{\delta|N}r_{\delta}\in \mathbb{Z}$, 
with the additional properties that
	$$\sum_{\delta|N} \delta r_{\delta}\equiv 0 \pmod{24}$$ and
	$$\sum_{\delta|N} \frac{N}{\delta}r_{\delta}\equiv 0 \pmod{24},$$
	then $f(z)$ satisfies $$f\left( \frac{az+b}{cz+d}\right)=\chi(d)(cz+d)^{\ell}f(z)$$
	for every  $\begin{bmatrix}
		a  &  b \\
		c  &  d      
	\end{bmatrix} \in \Gamma_0(N)$. Here the character $\chi$ is defined by $\chi(d):=\left(\frac{(-1)^{\ell} \prod_{\delta |N}\delta^{r_{\delta}}}{d}\right)$. 
	 In addition, if $c, d,$ and $N$ are positive integers with $d|N$ and $\gcd(c, d)=1$, then the order of vanishing of $f(z)$ at the cusp $\frac{c}{d}$ 
	is $\frac{N}{24}\sum_{\delta|N}\frac{\gcd(d,\delta)^2r_{\delta}}{\gcd(d,\frac{N}{d})d\delta}$.
\end{thm}
Suppose that $\ell$ is a positive integer and that $f(z)$ is an eta-quotient satisfying the conditions of the above theorem. 
If $f(z)$ is holomorphic at all of the cups of $\Gamma_0(N)$, 
then $f(z)\in M_{\ell}(\Gamma_0(N), \chi)$. We now use this fact for the eta-quotient $B_k$ defined by 
$$B_k(z)=\frac{\eta(48z)^{2^k-2}}{\eta(24z)^4\eta(96z)^{2^{k-1}-2}}.$$
Using Theorem \ref{thm_ono1}, we find that $B_k(z)\in M_{2^{k-2}-2}\left(\Gamma_0(384)\right)$ for $k \geq 4$. 
\par From \eqref{gen-ocp}, we can rewrite $\overline{B}(q^{24})$ as the following eta-quotient
$$\overline{B}(z)=\overline{B}(q^{24})=\frac{\eta(96z)^2}{\eta(24z)^4\eta(48z)^2}.$$
Let $$F_k(z)=\frac{\eta(48z)^{2^k}}{\eta(96z)^{2^{k-1}}}.$$
Then, we have $$B_k(z)=\frac{\eta(48z)^{2^k-2}}{\eta(24z)^4\eta(96z)^{2^{k-1}-2}}=\overline{B}(z)F_k(z).$$
It is not hard to establish the fact that $F_k(z)\equiv 1\pmod{2^k}$. Thus,
$$B_k(z)\equiv \overline{B}(z)\pmod{2^k}.$$
Now, let $m$ be a positive integer. From a deep theorem of Serre \cite[p. 43]{ono}, it follows that if $f(z)$ is an integral weight modular form in $M_k(\Gamma_0(N), \chi)$ 
which has a Fourier expansion 
	$$f(z)=\sum_{n=0}^{\infty}a(n)q^n\in \mathbb{Z}[[q]],$$
	then there is a constant $\alpha>0$  such that
	$$ \# \left\{n\leq X: a(n)\not\equiv 0 \pmod{m} \right\}= \mathcal{O}\left(\frac{X}{(\log{}X)^{\alpha}}\right).$$
Since $B_k(z)\in M_{2^{k-2}-2}\left(\Gamma_0(384)\right)$ for $k \geq 4$, the Fourier coefficients of $B_k(z)$ are almost always divisible by $2^k$ 
and so are the Fourier coefficients of $\overline{B}(z)$. Now, 
$$\overline{B}(z)=\overline{B}(q^{24})=\sum_{n=0}^{\infty}\overline{b}(n)q^{24n},$$ and hence $\overline{b}(n)$ is a multiple of $2^k$ for almost all $n$ and $k\geq 4$. 
This also trivially implies that 
$\overline{b}(n)$ is a multiple of $2^k$ for almost all $n$ and $k < 4$. This completes the proof of the Theorem \ref{thm3}.
\par We now prove Theorem \ref{thm4}. The generating function of $\overline{a}(n)$ is given by 
\begin{align*}
	\overline{A}(q):=\sum_{n=0}^{\infty}\overline{a}(n)q^n&=\frac{(-q; q)_{\infty}(-q^2; q^2)_{\infty}}{(q; q)_{\infty}(q^2; q^2)_{\infty}}.
\end{align*}
Let $$A_k(z)=\frac{\eta(48z)^{2^k-1}}{\eta(24z)^2\eta(96z)^{2^{k-1}-1}}.$$
Using Theorem \ref{thm_ono1}, we find that $A_k(z)\in M_{2^{k-2}-1}\left(\Gamma_0(768)\right)$ for $k > 2$. We also have 
$$A_k(z)=\overline{A}(z)F_k(z)\equiv \overline{A}(z) \pmod{2^k}.$$
Following the proof of Theorem \ref{thm3}, we now readily arrive at the desired result. This recompletes the proof of Theorem \ref{thm4}.

\end{document}